\documentclass[final]{jotart}
\usepackage{amsmath}
\usepackage{amssymb}
\usepackage{graphicx}
\theoremstyle{proclaim}
\newtheorem{theorem}{Theorem}[section]
\newtheorem{lemma}[theorem]{Lemma}
\newtheorem{corollary}[theorem]{Corollary}
\newtheorem{proposition}[theorem]{Proposition}
\theoremstyle{statement}
\newtheorem{remark}[theorem]{Remark}
\newtheorem{definition}[theorem]{Definition}

\theoremstyle{fancyproclaim}

\numberwithin{equation}{section}

\begin{document}
\issueinfo{00}{0}{0000} 
\commby{Editor}
\pagespan{101}{110}
\date{Month dd, yyyy}
\revision{Month dd, yyyy}
\title[Von Neumann Type of Trace Inequalities for Schatten-Class Operators]{Von Neumann Type of Trace Inequalities for Schatten-Class Operators}
\author[Gunther Dirr {\protect\and} Frederik vom Ende]{Gunther Dirr {\protect
\and} Frederik vom Ende}
\address{GUNTHER DIRR, Institute of Mathematics, University of W\"urzburg, D-97074 W{\"u}rzburg, Germany}
\email{dirr@mathematik.uni-wuerzburg.de}
\address{FREDERIK VOM ENDE, Department of Chemistry, Technische Universit\"at M\"unchen, D-85747 Garching, Germany--and--Munich Centre for Quantum Science and Technology (MCQST), D-80799~M{\"u}nchen, Germany}
\email{frederik.vom-ende@tum.de}
\begin{abstract} 
We generalize von Neumann's well-known trace inequality, as well as related eigenvalue inequalities
for hermitian matrices, to Schatten-class operators between complex Hilbert spaces of infinite dimension.
To this end, we exploit some recent results on the $C$-numerical range of Schatten-class operators
. For the readers' convenience, we sketched the proof of these results in the
Appendix.
\end{abstract}
\maketitle

\section{INTRODUCTION}

In the mid thirties of the last century, von Neumann \cite[Thm.~1]{NEUM-37} derived the 
following beautiful and widely used trace inequality for complex $n \times n$ matrices:

\medskip
{ \it
Let $A,B\in\mathbb C^{n\times n}$ with singular values $s_1(A)\geq s_2(A)\geq\ldots\geq s_n(A)$ and 
$s_1(B)\geq s_2(B)\geq\ldots\geq s_n(B)$, respectively, be given. Then
\begin{equation}\label{eq:von_Neumann}
\max_{U,V \in \;\mathcal U_n}|\operatorname{tr}(AUBV)|=\sum\nolimits_{j=1}^ns_j(A)s_j(B)\,,
\end{equation}
where $\mathcal U_n$ denotes the unitary group.}

\medskip
\noindent
In fact, the above result can be reinterpreted as a characterization of the image of the unitary 
double-coset $\{AUBV \,|\,U,V\in\mathbb C^{n\times n}\text{ unitary}\}$ under the trace-functional,
i.e.
\begin{equation}\label{eq:von_Neumann-2}
\{\operatorname{tr}(AUBV)\,|\, U,V \in \mathcal U_n\} = K_r(0)
\end{equation}
with $r := \sum_{j=1}^ns_j(A)s_j(B)$ and $K_r(0)=\{z\in\mathbb C\,,\,|z|\leq r\}$ being the closed disk 
of radius $r$ centred around the origin. This results from the elementary observation that the left-hand
side of \eqref{eq:von_Neumann-2} is circular (simply replace $U$ by $e^{ i \varphi}U$). Another 
well-known consequence of \eqref{eq:von_Neumann}, a von Neumann inequality for hermitian matrices 
\cite[Ch.~9.H.1]{MarshallOlkin}, reads as follows.

\medskip
{ \it
Let $A,B\in\mathbb C^{n\times n}$ hermitian with respective eigenvalues $(\lambda_j(A))_{j=1}^n$ and 
$(\lambda_j(B))_{j=1}^n$ be given. Then
\begin{equation}\label{eq:von_Neumann-hermitian}
\sum\nolimits_{j=1}^n \lambda_j^\downarrow(A)\lambda_j^\uparrow(B) 
\leq \operatorname{tr}(AB)\leq 
\sum\nolimits_{j=1}^n \lambda_j^\downarrow(A)\lambda_j^\downarrow(B)\,,
\end{equation}
where the superindeces $\downarrow$ and $\uparrow$ denote the decreasing and increasing 
sorting of the eigenvalue vectors, respectively.} 

\medskip
\noindent

The area of applications of von Neumann's inequalities and, more generally, singular value 
decompositions (SVD) is enormous. It ranges from operator theory \cite{KyFan51,Schatten50} and 
numerics \cite{GolubLoan} to more applied fields like control theory \cite{HM94},
neural networks \cite{Oja89} as well as quantum dynamics and quantum control \cite{Science98,OLE-95}.
An overview can be found in \cite{MarshallOlkin, Mirsky75}. Now the goal of this short contribution
is to generalize these inequalities to Schatten-class operators on infinite-dimensional Hilbert
spaces. In doing so, some recent results on the $C$-numerical range of Schatten-class
operators \cite{dirr_ve,dirr_ve_schatten} turn out to be quite helpful. For the readers'
convenience, we sketched the corresponding proofs in Appendix \ref{app_A}.

\medskip

This paper is organized as follows: Section \ref{sec:prelim} introduces the key notions and concepts
of this work such as \ref{sec:schatten} Schatten classes, \ref{sec:set_conv} convergence of compact sets
via the Hausdorff metric as well as \ref{sec:c_num_schatten} the $C$-numerical range for Schatten-class
operators. Section \ref{sec:results} then presents the main results as mentioned
above. Appendix \ref{app_A} outlines the outsourced proof of some crucial geometrical results regarding
the $C$-numerical range.

\section{NOTATION AND PRELIMINARIES}\label{sec:prelim}
Unless stated otherwise, here and henceforth $\mathcal X$ and $\mathcal Y$ are arbitrary 
infinite-dimensional complex Hilbert spaces while $\mathcal H$ and $\mathcal G$ are reserved
for infinite-dimensional {\em separable} complex Hilbert spaces. 
Moreover, let $\mathcal B(\mathcal X,\mathcal Y)$, $\mathcal U(\mathcal X,\mathcal Y)$, 
$\mathcal K(\mathcal X,\mathcal Y)$, $\mathcal F(\mathcal X,\mathcal Y)$ and 
$\mathcal B^p(\mathcal X,\mathcal Y)$ denote the set of all bounded, unitary, compact, finite-rank and 
$p$-th Schatten-class operators between $\mathcal X$ and $\mathcal Y$, respectively. As usual, if 
$\mathcal X$ and $\mathcal Y$ coincide we simply write $\mathcal B(\mathcal X)$, $\mathcal U(\mathcal X)$, 
etc.

Scalar products are conjugate linear in the first argument and linear in the second one. For an arbitrary
subset $S \subset \mathbb{C} $, the notations $\overline{S}$ and $\operatorname{conv}(S)$ stand for its 
closure and convex hull, respectively. Finally, given $p,q\in[1,\infty]$, we say $p$ and $q$ are conjugate
if $\frac1p+\frac1q=1$.

\subsection{INFINITE-DIMENSIONAL HILBERT SPACES AND THE SCHATTEN CLASSES}\label{sec:schatten}
For a comprehensive introduction to Hilbert spaces of infinite dimension as well as Schatten-class operators, we refer
to, e.g., \cite{berberian1976,MeiseVogt97en} and \cite{Dunford63}. Here, we recall only some basic results which 
will be used frequently throughout this paper. 
%
\begin{lemma}[Schmidt decomposition]\label{thm_1}
For each $C \in \mathcal K(\mathcal X,\mathcal Y)$, there exists a decreasing null sequence 
$(s_n(C))_{n\in\mathbb N}$ in $[0,\infty)$ as well as orthonormal systems $(f_n)_{n\in\mathbb N}$ in $\mathcal X$
and $(g_n)_{n\in\mathbb N}$ in $\mathcal Y$ such that
\begin{equation}\label{eq:Schmidt}
C = \sum\nolimits_{n=1}^\infty s_n(C)\langle f_n,\cdot\rangle g_n\,,
\end{equation}
where the series converges in the operator norm.
\end{lemma}
As the \emph{singular numbers} $(s_n(C))_{n\in\mathbb N}$ in Lemma \ref{thm_1} are uniquely determined by $C$,
the \emph{$p$-th Schatten-class} $\mathcal B^p(\mathcal X,\mathcal Y)$ is (well-)defined via
\begin{align*}
\mathcal B^p(\mathcal X,\mathcal Y)
:= \Big\lbrace C \in\mathcal K(\mathcal X,\mathcal Y)\,\Big|\,\sum\nolimits_{n=1}^\infty s_n(C)^p<\infty\Big\rbrace
\end{align*}
for $p\in [1,\infty)$. The Schatten-$p$-norm
\begin{align*}
\|C\|_p := \Big(\sum\nolimits_{n=1}^\infty s_n(C)^p\Big)^{1/p}
\end{align*}
turns $\mathcal B^p(\mathcal X,\mathcal Y)$ into a Banach space. Moreover, for $p=\infty$, we identify
$\mathcal B^\infty (\mathcal X,\mathcal Y)$ with the set of all compact operators 
$\mathcal K(\mathcal X,\mathcal Y)$ equipped with the norm
\begin{align*}
\|C\|_\infty := \sup_{n \in \mathbb{N}}s_n(C) = s_1(C)\,.
\end{align*}
Note that $\|C\|_\infty$ coincides with the ordinary operator norm $\|C\|$. Hence 
$\mathcal B^\infty (\mathcal X,\mathcal Y)$ constitutes a closed subspace of 
$\mathcal B (\mathcal X,\mathcal Y)$ and thus a Banach space, too. 

\begin{remark}\label{rem_Schatten_p}
Evidently, if $C\in\mathcal B^p(\mathcal X,\mathcal Y)$ for some $p\in[1,\infty]$ then the series
\eqref{eq:Schmidt} converges in the Schatten-$p$-norm. 
\end{remark}

The following results can be found in \cite[Coro.~XI.9.4 \& Lemma XI.9.9]{Dunford63}.

\begin{lemma}\label{lemma_10}
\begin{itemize}
\item[(a)] Let $p\in [1,\infty]$. Then for all $S,T\in\mathcal B(\mathcal X)$, $C\in\mathcal B^p(\mathcal X)$:
\begin{align*}
\|SCT\|_p\leq\Vert S\Vert\|C\|_p\Vert T\Vert\,.
\end{align*}
\item[(b)] Let $1\leq p\leq q\leq\infty$. Then
$\mathcal B^p(\mathcal X,\mathcal Y)\subseteq\mathcal B^q(\mathcal X,\mathcal Y)$ and 
$\|C\|_p\geq\|C\|_q$ for all $C\in\mathcal B^p(\mathcal X,\mathcal Y)$.
\end{itemize}
\end{lemma}
\noindent Note that due to (a), all Schatten-classes $\mathcal B^p(\mathcal X)$ constitute--just
like the compact operators--a two-sided ideal in the $C^*$-algebra of all bounded operators
$\mathcal B(\mathcal X)$.\medskip

Now for any $C\in\mathcal B^1(\mathcal X)$, the trace of $C$ is defined via
\begin{align}\label{eq:trace}
\operatorname{tr}(C):=\sum\nolimits_{i\in I}\langle f_i,Cf_i\rangle\,,
\end{align}
where $(f_i)_{i\in I}$ can be any orthonormal basis of $\mathcal X$. The trace is well-defined, as
one can show that the right-hand side of \eqref{eq:trace} is finite and does not depend on the choice of 
$(f_i)_{i\in I}$. Important properties are the following, cf. \cite[Lemma XI.9.14]{Dunford63}.

\begin{lemma}\label{lemma_nu_hoelder}
Let $C\in\mathcal B^p(\mathcal X,\mathcal Y)$ and $T\in\mathcal B^q(\mathcal Y,\mathcal X)$ with 
$p,q\in [1,\infty]$ conjugate. Then one has $CT\in\mathcal B^1(\mathcal Y)$ and $TC\in\mathcal B^1(\mathcal X)$
with
$$
\operatorname{tr}(CT)=\operatorname{tr}(TC)\quad\text{and}\quad |\operatorname{tr}(CT)|\leq \|C\|_p\|T\|_q\,.
$$
\end{lemma}

In order to recap the well-known diagonalization result for compact normal operators, we first have to
fix the term \emph{eigenvalue sequence} of a compact operator $T \in \mathcal K(\mathcal H)$. In general,
it is obtained by arranging the (necessarily countably many) non-zero eigenvalues in decreasing order
with respect to their absolute value and each eigenvalue is repeated as many times as its algebraic
multiplicity\footnote{By \cite[Prop.~15.12]{MeiseVogt97en}, every non-zero element $\lambda \in \sigma(T)$
of the spectrum of $T$ is an eigenvalue of $T$ and has a well-defined finite algebraic multiplicity 
$\nu_a(\lambda)$, e.g., $\nu_a(\lambda) := \dim \ker (T - \lambda I)^{n_0}$, where $n_0 \in \mathbb N$ 
is the smallest natural number $n \in \mathbb N$ such that $\ker (T - \lambda I)^n = \ker (T - \lambda I)^{n+1}$.
\label{footnote_alg_mult}} calls for. If only finitely many non-vanishing eigenvalues exist,
then the sequence is filled up with zeros, see \cite[Ch.~15]{MeiseVogt97en}. For our purposes, we have to 
pass to a slightly \emph{modified eigenvalue sequence} as follows: 

\begin{itemize}
\item 
If the range of $T$ is infinite-dimensional and the kernel of $T$ is finite-dimensional,
then put $\operatorname{dim}(\operatorname{ker}T)$ zeros at the beginning of the eigenvalue
sequence of $T$. \vspace{4pt}
\item 
If the range and the kernel of $T$ are infinite-dimensional, mix infinitely many 
zeros into the eigenvalue sequence of $T$.

Because in Definition
\ref{defi_3} arbitrary permutations will be applied to the modified eigenvalue sequence, we do not
need to specify this mixing procedure further, cf. also \cite[Lemma 3.6]{dirr_ve}.\vspace{4pt}
\item
If the range of $T$ is finite-dimensional leave the eigenvalue sequence of $T$ unchanged. 
\end{itemize}

\begin{lemma}[\cite{berberian1976}, Thm.~VIII.4.6]\label{lemma_berb_4_6}
Let $T\in\mathcal K(\mathcal H)$ be normal, i.e.~$T^\dagger T=TT^\dagger$. Then there exists an 
orthonormal basis $(e_n)_{n\in\mathbb N}$ of $\mathcal H$ such that
\begin{align*}
T=\sum\nolimits_{j=1}^\infty \lambda_j(T)\langle e_j,\cdot\rangle e_j
\end{align*}
where $(\lambda_j(T))_{j\in\mathbb N}$ is the modified eigenvalue sequence of $T$.
\end{lemma}

\subsection{SET CONVERGENCE}\label{sec:set_conv}

In order to transfer results about convexity and star-shapedness of the $C$-numerical
range from matrices to Schatten-class operators, we need a concept of set convergence. 
We will use the Hausdorff metric on compact subsets (of $\mathbb C$) and the associated notion
of convergence, see, e.g., \cite{nadler1978}.\medskip

The distance between $z \in \mathbb C$ and a non-empty compact subset $A \subseteq \mathbb C$ is
given by $d(z,A) := \min_{w \in A} d(z,w) = \min_{w \in A} |z-w|$, based on which the \emph{Hausdorff metric} $\Delta$ on the set of all non-empty
compact subsets of $\mathbb C$ is defined via 
\begin{align*}
\Delta(A,B) := \max\Big\lbrace \max_{z \in A}d(z,B),\max_{z \in B}d(z,A) \Big\rbrace\,.
\end{align*}
The following characterization of the Hausdorff metric is readily verified.

\begin{lemma}\label{lemma_11}
Let $A,B \subset \mathbb C$ be two non-empty compact sets and let $\varepsilon > 0$. 
Then $\Delta(A,B) \leq \varepsilon$ if and only if for all $z \in A$, there exists $w \in B$ 
with $d(z,w) \leq \varepsilon$ and vice versa.
\end{lemma}

With this metric one can introduce the notion of convergence for sequences $(A_n)_{n\in\mathbb N}$ 
of non-empty compact subsets of $\mathbb C$ such that the maximum- as well as the minimum-operator
are continuous in the following sense.

\begin{lemma}\label{lemma_lim_max}
Let $(A_n)_{n\in\mathbb N}$ be a bounded sequence of non-empty, compact subsets of $\mathbb R$ 
which converges to $A\subset\mathbb R$. Then the sequences of real numbers $(\max A_n)_{n\in\mathbb N}$
and $(\min A_n)_{n\in\mathbb N}$ are convergent with 
\begin{align*}
\lim_{n\to\infty}(\max A_n) 
=\max A \quad\text{and}\quad \lim_{n\to\infty}(\min A_n)=\min A\,.
\end{align*}
\end{lemma}
\begin{proof}
Let $\varepsilon>0$. By assumption, there exists $N\in\mathbb N$ such that $\Delta(A_n,A)<\varepsilon$ 
for all $n\geq N$. Hence by Lemma \ref{lemma_11} one finds $a_n\in A_n$ with 
$|\max A - a_n| < \varepsilon$ and thus
$
\max A < a_n + \varepsilon < \max A_n + \varepsilon\,.
$
Similarly, there exists $a \in A$ such that $|\max A_n - a| < \varepsilon$, so
$
\max A_n < a + \varepsilon < \max A + \varepsilon\,.
$
Combining both estimates, we get $|\max A - \max A_n| < \varepsilon$. The case of the minimum 
is shown analogously.
\end{proof}

\subsection{THE $C$-NUMERICAL RANGE OF SCHATTEN-CLASS OPERATORS}\label{sec:c_num_schatten}

In this subsection, we present a few approximation results and collect some material on the 
$C$-numerical range of Schatten-class operators which is of fundamental importance in Section \ref{sec:results}. Because said results 
appeared only in an addendum \cite{dirr_ve_schatten} 
to another publication \cite{dirr_ve} on trace-class operators, we decided to sketch the proof 
in the appendix for the readers' convenience.

\begin{definition}\label{defi_1}
Let $p,q\in [1,\infty]$ be conjugate. Then for $C\in\mathcal B^p(\mathcal X)$
and $T\in\mathcal B^q(\mathcal X)$, the \emph{$C$-numerical range of} $T$ is defined to be
\begin{align*}
W_C(T):=\lbrace \operatorname{tr}(CU^\dagger TU)\,|\,U\in\mathcal U(\mathcal X) \rbrace\,.
\end{align*}
Following \eqref{eq:von_Neumann-2}, for $C\in\mathcal B^p(\mathcal X,\mathcal Y)$ and $T\in\mathcal B^q(\mathcal Y,\mathcal X)$
with $p,q\in [1,\infty]$ conjugate one may actually introduce the more general set (now invoking the unitary 
equivalence orbit $UTV$ of $T$ instead of the unitary similarity orbit $U^\dagger T U$)
\begin{align*}
S_C(T):= \{\operatorname{tr}(CUTV)\,|\,U\in\mathcal U(\mathcal X)\,, V\in\mathcal U(\mathcal Y)\}\,.
\end{align*}
\end{definition}
\noindent Note that all traces involved are well-defined due to Lemma \ref{lemma_10} and 
\ref{lemma_nu_hoelder}.

\begin{lemma}\label{lemma_proj_strong_conv}
Let $p\in[1,\infty]$, $C \in \mathcal B^p(\mathcal X)$ and $(S_n)_{n\in\mathbb N}$ be a sequence in 
$\mathcal B(\mathcal X)$ which converges strongly to $S \in\mathcal B(\mathcal X)$. Then one has
$S_n C \to SC$, $CS_n^\dagger \to CS^\dagger$, and $S_nCS_n^\dagger \to SCS^\dagger$ for $n \to \infty$
with respect to the norm $\|\cdot\|_p$.
\end{lemma}

\begin{proof}
The cases $p=1$ and $p=\infty$ are proven in \cite[Lemma 3.2]{dirr_ve}. As the proof for $p\in(1,\infty)$
is essentially the same, we sketch only the major differences. First, choose $K\in\mathbb N$ such that
\begin{align*}
\sum\nolimits_{k=K+1}^\infty s_k(C)^p<\frac{\varepsilon^p}{(3\kappa)^p}\,,
\end{align*}
where $\kappa > 0$ satisfies $\|S\|\leq\kappa$ and $\|S_n\|\leq\kappa$ for all $n\in\mathbb N$. 
The existence of the constant $\kappa > 0$ is guaranteed by the uniform boundedness principle.
Then decompose $C=\sum\nolimits_{k=1}^\infty s_k(C)\langle e_k,\cdot\rangle f_k$ into $C = C_1 + C_2$ with 
$C_1 := \sum\nolimits_{k=1}^K s_k(C)\langle e_k,\cdot\rangle f_k$ finite-rank. By Lemma \ref{lemma_10}
one has
\begin{align*}
\|SC - S_nC\|_p&\leq \|SC_1-S_n C_1\|_p+\Vert S\Vert\|C_2\|_p+\Vert S_n\Vert\|C_2\|_p
\\
&<\|SC_1-S_nC_1 \|_p+\frac{2\varepsilon}{3}\,.
\end{align*}
Thus, what remains is to choose $N\in\mathbb N$ such that $\|SC_1-S_nC_1 \|_p<\varepsilon/3$
for all $n\geq N$. To this end, consider the estimate
\begin{align*}
\|SC_1-S_n C_1\|_p\leq  \sum_{k=1}^K s_k(C)\|\langle e_k,\cdot\rangle (Sf_k-S_nf_k)\|_p=\sum_{k=1}^K s_k(C) \Vert Sf_k-S_nf_k \Vert\,.
\end{align*}
Then the strong convergence of $(S_n)_{n\in\mathbb N}$ yields $N \in \mathbb N$ such that
\begin{align*}
\Vert Sf_k - S_nf_k \Vert<\frac{\varepsilon}{3\sum\nolimits_{k=1}^Ks_k(C)}
\end{align*}
for $k = 1, \dots, K$ and all $n\geq N$. This shows $\|SC - S_nC\|_p\to 0$ as $n\to\infty$. All
other assertions are an immediate consequence of $\|A\|_p = \|A^\dagger\|_p$ for 
$A \in \mathcal B^p(\mathcal X)$ and 
\begin{align*}
\|SCS^\dagger - S_nCS_n^\dagger\|_p & \leq \Vert S\Vert\|CS^\dagger-CS_n^\dagger\|_p+\|SC-S_nC\|_p\Vert S_n\Vert \\
& \leq \kappa \big(\|CS^\dagger-CS_n^\dagger\|_p + \|SC-S_nC\|_p\big)\,.\qedhere
\end{align*}
\end{proof}

\begin{proposition}\label{prop_1}
Let $C\in\mathcal B^p(\mathcal X,\mathcal Y)$, $T\in\mathcal B^q(\mathcal Y,\mathcal X)$ with 
$p,q\in [1,\infty]$ conjugate and let $(C_n)_{n\in\mathbb N}$ and $(T_n)_{n\in\mathbb N}$ be sequences in $\mathcal B^p(\mathcal X,\mathcal Y)$ and $\mathcal B^q(\mathcal Y,\mathcal X)$, respectively, such that
$
\lim_{n\to\infty}\|C-C_n\|_p=\lim_{n\to\infty}\|T-T_n\|_q=0\,.
$
Then
\begin{equation}\label{eq:S_C_n}
\lim_{n\to\infty}\overline{S_{C_n}(T_n)}=\overline{S_C(T)}\,.
\end{equation}
If, additionally, $\mathcal X=\mathcal Y$ then
\begin{align}\label{eq:W_C_n}
\lim_{n\to\infty}\overline{W_{C_n}(T_n)}=\overline{W_C(T)}\,.
\end{align}
\end{proposition}

\begin{proof}
W.l.o.g.~let $C_n,T_n\neq 0$ for some $n\in\mathbb N$--else all the involved sets would be trivial--so we may introduce the positive but (as seen via the reverse triangle inequality) finite numbers
$$
\kappa := \sup\{\|C\|_p,\|C_1\|_p,\|C_2\|_p,\ldots\}  
\quad\text{ and }\quad \tau :=\sup\{\|T\|_q,\|T_1\|_q,\|T_2\|_q,\ldots\}\,.
$$
Let $\varepsilon>0$. By assumption there exists $N\in\mathbb N$ such that
$$
\|C-C_n\|_p<\frac{\varepsilon}{4\tau}  \qquad\text{ as well as }\qquad \|T-T_n\|_q<\frac{\varepsilon}{4\kappa}
$$
for all $n\geq N$. We shall first tackle \eqref{eq:S_C_n}, as \eqref{eq:W_C_n} can be shown in complete analogy. 
The goal will be to satisfy the assumptions of Lemma \ref{lemma_11} in order to show
$\Delta(\overline{S_C(T)},\overline{S_{C_n}(T_n)})<\varepsilon$ for all $n\geq N$. \medskip

\noindent Let $w\in\overline{S_C(T)}$ so one finds $U\in\mathcal U(\mathcal X)$, $V\in\mathcal U(\mathcal Y)$
such that $w':=\operatorname{tr}(CUTV)$ satisfies $|w-w'|<\frac{\varepsilon}{2}$. Thus
for $w_n:=\operatorname{tr}(C_nUT_nV)$ by Lemma \ref{lemma_10} and \ref{lemma_nu_hoelder}
\begin{align*}
|w-w_n|&\leq|w-w'|-|w'-w_n|\\
&<\frac{\varepsilon}{2}+|\operatorname{tr}( (C-C_n)UTV )|+|\operatorname{tr}(VC_nU(T-T_n))|\\
&\leq\frac{\varepsilon}{2}+\|C-C_n\|_p\|U\|\|T\|_q\|V\|+\|V\|\|C_n\|_p\|U\|\|T-T_n\|q\\
&\leq \frac{\varepsilon}{2}+\|C-C_n\|_p\,\tau+\kappa\,\|T-T_n\|_q<\varepsilon
\end{align*}
for all $n \geq N$. 

Similarly, let $n \geq N$. Then for $v_n\in\overline{S_{C_n}(T_n)}$ 
one finds $U_n\in\mathcal U(\mathcal X)$, $V_n\in\mathcal U(\mathcal Y)$ such that 
$v_n':=\operatorname{tr}(C_nU_nT_nV_n)$ satisfies $|v_n-v_n'|<\frac{\varepsilon}{2}$. Thus for 
$\tilde v_n:=\operatorname{tr}(CU_nTV_n)$ we obtain
\begin{align*}
|v_n-\tilde v_n|&\leq|v_n-v_n'|-|v_n'-\tilde v_n|\\
&<\frac{\varepsilon}{2}+|\operatorname{tr}( (C_n-C)U_nT_nV_n )|+|\operatorname{tr}(V_nCU_n(T_n-T))|\\
&\leq \frac{\varepsilon}{2}+\|C-C_n\|_p\,\tau+\kappa\,\|T-T_n\|_q<\varepsilon\,.\qedhere
\end{align*}
\end{proof}

The preceding proposition together with Lemma \ref{lemma_proj_strong_conv} immediately 
entails the next result.

\begin{corollary}\label{lemma_2}
Let $C\in\mathcal B^p(\mathcal H)$, $T\in\mathcal B^q(\mathcal H)$ with $p,q\in [1,\infty]$ conjugate. 
Then 
$
\lim_{k\to\infty}\overline{W_C(\Pi_kT\Pi_k)}=\overline{W_C(T)}\,,
$
where $\Pi_k$ is the orthogonal projection onto the span of the first $k$ elements of 
an arbitrarily chosen orthonormal basis $(e_n)_{n\in\mathbb N}$ of $\mathcal H$.
\end{corollary}
\noindent Here we used the well-known fact that the orthogonal projections $\Pi_k$
strongly converge to the identity $\operatorname{id}_{\mathcal H}$ for $k\to\infty$, cf., e.g., 
\cite[Lemma 3.2]{dirr_ve}.
\begin{definition}[$C$-spectrum]\label{defi_3}
Let $p,q\in[1,\infty]$ be conjugate. Then, for $C\in\mathcal B^p(\mathcal H)$ with modified
eigenvalue sequence $(\lambda_n(C))_{n\in\mathbb N}$ and $T\in\mathcal B^q(\mathcal H)$ with modified
eigenvalue sequence $(\lambda_n(T))_{n\in\mathbb N}$, the $C$-\emph{spectrum of} $T$ is defined via
\begin{align*}
P_C(T):=\Big\lbrace \sum\nolimits_{n=1}^\infty \lambda_n(C)\lambda_{\sigma(n)}(T) \,\Big|\, 
\sigma:\mathbb N \to\mathbb N \text{ is any permutation}\Big\rbrace.
\end{align*}
\end{definition}

\noindent 
H\"older's inequality and the standard estimate
$\sum\nolimits_{n=1}^\infty  |\lambda_n(C)|^p \leq \sum\nolimits_{n=1}^\infty s_n(C)^p$,
cf.~\cite[Prop.~16.31]{MeiseVogt97en}, yield
$$
\sum\nolimits_{n=1}^\infty |\lambda_n(C)\lambda_{\sigma(n)}(T)|
\leq \Big(\sum\nolimits_{n=1}^\infty  s_n(C)^p \Big)^{1/p}\Big(\sum\nolimits_{n=1}^\infty s_n(T)^q\Big)^{1/q}
=\|C\|_p\|T\|_q\,,
$$
showing that the elements of $P_C(T)$ are well-defined and bounded by $\|C\|_p\|T\|_q$.

Now, if the operators $C$ and $T$ are particularly  ``nice'', one can connect the 
$C$-numerical range and the $C$-spectrum of $T$ as follows:

\begin{theorem}[\cite{dirr_ve_schatten}
]
\label{theorem_3}
Let $C\in\mathcal B^p(\mathcal H)$ and $T\in\mathcal B^q(\mathcal H)$ with $p,q\in [1,\infty]$ 
conjugate. Then the following statements hold.
\begin{itemize}
\item[(a)] $\overline{W_C(T)}$ is star-shaped with respect to the origin.\vspace{4pt}
\item[(b)] If either $C$ or $T$ is normal with collinear eigenvalues, then $\overline{W_C(T)}$ is convex.\vspace{4pt}
\item[(c)] If $C$ and $T$ both are normal, then $P_C(T)\subseteq W_C(T)\subseteq\operatorname{conv}(\overline{P_C(T)})$. If, in addition, the eigenvalues of $C$ or $T$ are collinear then $\overline{W_C(T)}=\operatorname{conv}(\overline{P_C(T)})$.
\end{itemize}
\end{theorem}
\noindent As stated in the beginning, a sketch of the proof can be found in Appendix \ref{app_A}.
%
%
%

\section{MAIN RESULTS}\label{sec:results}

Considering the inequalities \eqref{eq:von_Neumann} and \eqref{eq:von_Neumann-hermitian} from the introduction, it arguably is 
easier to generalize the former, i.e.~to generalize von Neumann's ``original'' trace inequality to Schatten-class 
operators. To start with we first investigate
the finite-rank case.

\begin{lemma}\label{lemma_S_C_finite_rank}
Let $C\in\mathcal F(\mathcal X,\mathcal Y)$, $T\in\mathcal F(\mathcal Y,\mathcal X)$ and $k:=\max\{\operatorname{rk}(C),\operatorname{rk}(T)\}\in\mathbb N_0$. Then $S_C(T)=K_r(0)$ where $r:=\sum_{j=1}^ks_j(C)s_j(T)$. 
\end{lemma}
\begin{proof}
Defining $k$ as above, Lemma \ref{thm_1} yields orthonormal systems $(e_j)_{j=1}^k$, $(h_j)_{j=1}^k$ in $\mathcal X$ and $(f_j)_{j=1}^k$, $(g_j)_{j=1}^k$ in $\mathcal Y$ such that
$$
C=\sum\nolimits_{j=1}^k s_j(C)\langle e_j,\cdot\rangle f_j\quad\text{ and }\quad T=\sum\nolimits_{j=1}^k s_j(T)\langle g_j,\cdot\rangle h_j\,.
$$
Note that forcing both sums to have same summation range means that, potentially,
some of the singular values have to be complemented by zeros, which is not of further
importance.\medskip

``$\subseteq$'': Let any $U\in\mathcal U(\mathcal X)$, $V\in\mathcal U(\mathcal Y)$ be given. Then
\begin{align*}
\operatorname{tr}(CUTV)&=\operatorname{tr}\Big( \sum\nolimits_{i,j=1}^{k}s_i(T)s_j(C)\langle e_j,Uh_i\rangle\langle V^\dagger g_i,\cdot\rangle f_j\Big)\\
&=\hphantom{\operatorname{tr}\Big(} \sum\nolimits_{i,j=1}^{k}s_i(T)s_j(C)\langle e_j,Uh_i\rangle\langle g_i,V f_j\rangle
\end{align*}
by direct computation. Now consider the subspaces
\begin{align*}
Z_1&:=\operatorname{span}\{e_1,\ldots,e_k,Uh_1,\ldots,Uh_k\}\subset \mathcal X\\
Z_2&:=\operatorname{span}\{f_1,\ldots,f_k,V^\dagger g_1,\ldots,V^\dagger g_k\}\subset\mathcal Y
\end{align*}
so there exist orthonormal bases of the form $$e_1,\ldots,e_k, e_{k+1},\ldots, e_N\quad\text{ and }\quad f_1,\ldots,f_k, f_{k+1},\ldots, f_{N'}$$ of $Z_1$ and $Z_2$ for some $N,N'\geq k$, respectively. W.l.o.g.\footnote{This can be done for example by sufficiently expanding the ``smaller'' orthonormal systems in $\mathcal X$ or $\mathcal Y$ and possibly passing to new subspaces $Z_1'\supset Z_1$ or $Z_2'\supset Z_2$ which is always doable because we are in infinite dimensions. The particular choice of $Z_1'$ and $Z_2'$ is irrelevant because we only need the 
orthonormal systems which represent $C$ and $T$ to be contained within these finite-dimensional subspaces.} we can assume $N=N'$ and define
$$
a_j:=(\langle e_l,U h_j\rangle)_{l=1}^N\in\mathbb C^N\quad\text{ and }\quad b_j:=(\langle f_l,V^\dagger g_j\rangle)_{l=1}^N\in\mathbb C^N
$$
for $j=1,\ldots,k$. This yields $N\times N$ matrices
\begin{align*}
C'=\operatorname{diag} (s_1(C),\ldots, s_k(C),0,\ldots,0)\quad\text{ and }\quad  T'=\sum\nolimits_{j=1}^k s_j(T) \langle b_j,\cdot\rangle a_j
\end{align*}
which satisfy $\operatorname{tr}(C'T')=\sum\nolimits_{i,j=1}^{k}s_i(T)s_j(C)\langle e_j,Uh_i\rangle\langle g_i,V f_j\rangle$. By construction, one readily verifies that $(a_j)_{j=1}^k,(b_j)_{j=1}^k$ are orthonormal systems in $\mathbb C^N$ so $s_j(T')=s_j(T)$ for all $j=1,\ldots,N$. Thus von Neumann's original result \eqref{eq:von_Neumann} yields
$$
|\operatorname{tr}(CUTV)|=|\operatorname{tr}(C'T')|\leq\sum\nolimits_{j=1}^N s_j(C')s_j(T')=\sum\nolimits_{j=1}^k s_j(C) s_j(T)\,.
$$

``$\supseteq$'': We first consider unitary operators $U_T\in\mathcal B(\mathcal X)$, $V_T\in\mathcal B(\mathcal Y)$ 
such that $U_Th_j=e_j$ and $V_Tf_j=g_j$ for all $j=1,\ldots,k$. This is always possible by completing the 
respective orthonormal systems $(e_j)_{j=1}^k$, $\ldots$ to orthonormal bases $(e_j)_{j\in J}$, $\ldots$ 
which can then be transformed into each other via some unitary. This allows us to construct $\tilde T:=U_TTV_T=\sum\nolimits_{j=1}^k s_j(T)\langle f_j,\cdot\rangle e_j$ such that
$$
\operatorname{tr}(C\tilde U\tilde T \tilde V)=\sum\nolimits_{i,j=1}^N s_j(C)s_i(T)\langle e_j,\tilde Ue_i\rangle\langle f_i,\tilde Vf_j\rangle
$$
for any $\tilde U\in\mathcal U(\mathcal X)$, $\tilde V\in\mathcal U(\mathcal Y)$. Of course $S_C(T)=S_C(\tilde T)$ and the latter satisfies
%
%
%
\begin{itemize}
\item $r\in S_C(\tilde T)$: choose $\tilde U=\operatorname{id}_{\mathcal X}$, $\tilde V=\operatorname{id}_{\mathcal Y}$ and also\vspace{4pt}
\item $0\in S_C(\tilde T)$: choose $\tilde U$, $\tilde V$ as cyclic shift on the first $k$ basis elements, i.e.
$$
\tilde U:\mathcal X\to\mathcal X\,,\qquad  e_j\mapsto \begin{cases} e_{j+1}&j=1,\ldots,k-1\\ e_{1}&j=k\\ e_j&j\in J\setminus\{1,\ldots,k\} \end{cases}
$$
and similarly $\tilde V$ (on $\{f_1,\ldots,f_k\}$).\vspace{2pt}
\end{itemize}
Now because the unitary group $\mathcal U(\mathcal Y)$ on any Hilbert space $\mathcal Y$ is path-connected
\footnote{The standard argument for this goes as follows, cf.~\cite[Proof of Thm.~12.37]{Rudin91}: For every 
$U\in\mathcal U(\mathcal Y)$ there exists self-adjoint $Q\in\mathcal B(\mathcal Y)$ such that $U=\exp(iQ)$. 
Then $t\mapsto T(t):=\exp(itQ)$ is a continuous mapping of $[0,1]$ into $\mathcal U(\mathcal Y)$ with 
$T(0)=\operatorname{id}_{\mathcal Y}$ and $T(1)=U $. Thus every unitary operator is path-connected to the 
identity which implies path-connectedness of $\mathcal U(\mathcal Y$).} and because the mapping 
$f:\mathcal B(\mathcal X)\times\mathcal B(\mathcal Y)\to \mathbb C$, $(U,V)\mapsto \operatorname{tr}(CU\tilde TV)$ 
is continuous, 
the image $f(\mathcal U(\mathcal X)\times\mathcal U(\mathcal Y))$ has to 
be path-connected as well. In particular, $0$ and $r$ are path-connected within $S_C(T)$, i.e.~for every 
$s\in[0,r]$ there exists $\phi(s)\in[0,2\pi)$ such that $se^{i\phi(s)}\in S_C(\tilde T)=S_C(T)$. 

Finally, we can use the fact that $S_C(T)$ is circular--which follows easily by replacing  
$U$ by $e^{i\varphi} U\in\mathcal U(\mathcal X)$ with $\varphi\in [0,2\pi]$--to conclude $S_C(T)\supseteq K_r(0)$
and thus $S_C(T) = K_r(0)$.
\end{proof}

\begin{theorem}
Let $C\in\mathcal B^p(\mathcal X,\mathcal Y)$, $T\in\mathcal B^q(\mathcal Y,\mathcal X)$ with $p,q\in[1,\infty]$ conjugate. Then 
\begin{equation}\label{eq:von-Neumann-infinite-dim}
\sup_{U \in \;\mathcal U(X) , V \in \;\mathcal U(Y)}|\operatorname{tr}(CUTV)|
= \sum\nolimits_{j=1}^\infty s_j(C)s_j(T)\,.
\end{equation}
In particular, one has $\overline{S_C(T)}=K_r(0)$ with $r:=\sum_{j=1}^\infty s_j(C)s_j(T)$.
\end{theorem}

\begin{proof} 
By Lemma \ref{thm_1} $C=\sum_{j=1}^\infty s_j(C)\langle e_j,\cdot\rangle f_j$, 
$T=\sum_{j=1}^\infty s_j(T)\langle g_j,\cdot\rangle h_j$ for some orthonormal systems $(e_j)_{j\in\mathbb N}$,
$(h_j)_{j\in\mathbb N}$ in $\mathcal X$ and $(f_j)_{j\in\mathbb N}$, $(g_j)_{j\in\mathbb N}$ in $\mathcal Y$. This allows 
us to define finite rank approximations $C_n:=\sum_{j=1}^n  s_j(C)\langle e_j,\cdot\rangle f_j$
and $T_n:=\sum_{j=1}^n s_j(T)\langle g_j,\cdot\rangle h_j$ 
To pass to the original operators $C,T$, 
we use Remark \ref{rem_Schatten_p} to see
\begin{equation*}
\lim_{n\to\infty}\|C_n-C\|_p=0\qquad \text{and} \qquad\lim_{n\to\infty}\|T_n-T\|_q=0\,.
\end{equation*}
%
Because of this we may apply Proposition \ref{prop_1} and Lemma \ref{lemma_S_C_finite_rank} to obtain
$$
\overline{S_C(T)}=\lim_{n\to\infty}\overline{S_{C_n}(T_n)}=\lim_{n\to\infty} K_{r_n}(0)
$$
with $r_n:=\sum_{j=1}^n s_j(C)s_j(T)$. Using the obvious fact $\Delta(K_r(0),K_{r_n}(0))=|r-r_n|$ for all $n\in\mathbb N$ 
one readily verifies
$
\overline{S_C(T)}=\lim_{n\to\infty} K_{r_n}(0)=K_r(0)
$ with  $r=\sum_{j=1}^\infty s_j(C)s_j(T)$.
\end{proof}

\begin{remark}
To see that the supremum in \eqref{eq:von-Neumann-infinite-dim} is not necessarily a maximum, consider $\mathcal H=\ell_2(\mathbb N)$ with standard basis $(e_j)_{j\in\mathbb N}$. Now the positive definite trace-class operator $C=\sum_{j=1}^\infty \frac{1}{2^j}\langle e_j,\cdot\rangle e_j$ as well as the compact operator $T=\sum_{k=1}^\infty \frac{1}{2^k}\langle e_{k+1},\cdot\rangle e_{k+1}$ satisfy
$$
\operatorname{tr}(CUTV)=  \sum\nolimits_{j=1}^\infty \frac{1}{2^j}\langle e_j, UTV e_j\rangle=\sum\nolimits_{j,k=1}^\infty \frac{1}{2^j} \frac{1}{2^k}\langle e_j,Ue_{k+1}\rangle\langle e_{k+1},Ve_j\rangle
$$
for any $U,V\in\mathcal U(\mathcal H)$. We know that $\sup_{U,V\in\mathcal U(\mathcal H)}|\operatorname{tr}(CUTV)|=\sum_{j=1}^\infty(\frac{1}{2^j})^2$ but if this was a maximum, then by the above calculation $\langle e_j,Ue_{k+1}\rangle=\langle e_{k+1},Ve_j\rangle=\delta_{jk}$ for all $j,k\in\mathbb N$. The only operators which satisfy these conditions are 
the left- and the right-shift, respectively, both of which are not unitary--a contradiction.
\end{remark}

Finally, we are prepared to extend inequality \eqref{eq:von_Neumann-hermitian} to Schatten-class operators 
on separable Hilbert spaces.

\begin{theorem}
Let $C\in\mathcal B^p(\mathcal H)$, $T\in\mathcal B^q(\mathcal H)$ both be self-adjoint with $p,q$ conjugate and 
let the positive semi-definite operators $C^+,T^+$ and $C^-,T^-$ denote the positive and negative part of $C,T$, 
respectively (i.e.~$C=C^+-C^-$, $T=T^+-T^-$). Then
\begin{equation}\label{eq:sup_WCT}
\operatorname{sup}_{U\in\mathcal U(\mathcal H)}\operatorname{tr}(CU^\dagger TU)
= \sum\nolimits_{j=1}^\infty \big(\lambda_j^\downarrow (C^+)\lambda_j^\downarrow (T^+) 
+ \lambda_j^\downarrow (C^-)\lambda_j^\downarrow (T^-)\big)
\end{equation}
as well as
\begin{equation}\label{eq:inf_WCT}
\operatorname{inf}_{U\in\mathcal U(\mathcal H)}\operatorname{tr}(CU^\dagger TU)
=- \sum\nolimits_{j=1}^\infty \big(\lambda_j^\downarrow (C^+)\lambda_j^\downarrow (T^-)
+ \lambda_j^\downarrow (C^-)\lambda_j^\downarrow (T^+)\big)\,.
\end{equation}
 In particular, one has:\\
 \scalebox{0.87}{
\begin{minipage}{1.14\textwidth}
$$
-\sum_{j=1}^\infty \big(\lambda_j^\downarrow (C^+)\lambda_j^\downarrow (T^-)+ \lambda_j^\downarrow (C^-)\lambda_j^\downarrow (T^+)\big)
\leq \operatorname{tr}(CT)\leq 
\sum_{j=1}^\infty \big(\lambda_j^\downarrow (C^+)\lambda_j^\downarrow (T^+)+\lambda_j^\downarrow (C^-)\lambda_j^\downarrow (T^-)\big)
$$
\end{minipage}
}
\end{theorem}

\begin{proof}
Let $C\in\mathcal B^p(\mathcal H)$, $T\in\mathcal B^q(\mathcal H)$ both be self-adjoint with $p,q$ conjugate and first 
assume that $T$ has at most $k\in\mathbb N$ non-zero eigenvalues. Then the following is straightforward to show:
\begin{align*}
\max \operatorname{conv}(\overline{P_C(T)}) &=\hphantom{-} 
\sum\nolimits_{j=1}^k\lambda_j^\downarrow(C^+) \lambda_{j}^\downarrow(T^+) 
+ \sum\nolimits_{j=1}^k \lambda_j^\downarrow(C^-) \lambda_{j}^\downarrow(T^-)\\
\min \operatorname{conv}(\overline{P_C(T)}) &=- 
\sum\nolimits_{j=1}^k\lambda_j^\downarrow(C^+) \lambda_{j}^\downarrow(T^-) 
- \sum\nolimits_{j=1}^k \lambda_j^\downarrow(C^-) \lambda_{j}^\downarrow(T^+)
\end{align*}
Note that in this case the (modified) eigenvalue sequences of $T$ contains infinitely
many zeros. Now let us address the general case. Choose any orthonormal eigenbasis $(e_n)_{n\in\mathbb N}$ of $T$ 
with corresponding modified eigenvalue sequence (Lemma \ref{lemma_berb_4_6}). Moreover, let 
$\Pi_k=\sum\nolimits_{j=1}^k\langle e_j,\cdot\rangle e_j$ the projection onto the span
of the first $k$ eigenvectors of $T$. Then $\Pi_kT\Pi_k$ has at most $k$ non-zero eigenvalues and our 
preliminary considerations combined with Corollary  \ref{lemma_2} and
Theorem \ref{theorem_3} (c) as well as Lemma \ref{lemma_lim_max} readily imply
\begin{align*}
\sup_{U\in\mathcal U(\mathcal H)}&\operatorname{tr}(C U^\dagger TU)
= \max \overline{W_C(T)} = \max \lim_{k\to\infty} \overline{W_C(\Pi_kT\Pi_k)}\\
& =\lim_{k\to\infty}\max \overline{W_C(\Pi_kT\Pi_k)} 
= \lim_{k\to\infty}\max \operatorname{conv}(\overline{P_C(\Pi_kT\Pi_k)})\\
& =\lim_{k\to\infty}\Big(\sum\nolimits_{j=1}^k \lambda_j^\downarrow(C^+) \lambda_{j}^\downarrow(\Pi_kT^+\Pi_k)
+ \sum\nolimits_{j=1}^k \lambda_j^\downarrow(C^-) \lambda_{j}^\downarrow(\Pi_kT^-\Pi_k)\Big)
\end{align*}
where we used the identity $(\Pi_kT\Pi_k)^\pm = \Pi_kT^\pm\Pi_k$. Now, the last step is to show 
that $(\sum\nolimits_{j=1}^k \lambda_j^\downarrow(C^+)\lambda_j^\downarrow(\Pi_kT^+\Pi_k))_{k\in\mathbb N}$ converges to
$\sum\nolimits_{j=1}^\infty\lambda_j^\downarrow(C^+)\lambda_j^\downarrow(T^+)$. Let $\varepsilon>0$
(and w.l.o.g.~$T\neq 0$). As $(\lambda_j^\downarrow(C^+))_{j\in\mathbb N}$ is a sequence in $\ell^p_+(\mathbb N)$ 
we find $N\in\mathbb N$ with
$$
\Big(\sum\nolimits_{j=N+1}^\infty\big(\lambda_j^\downarrow(C^+)\big)^p\Big)^{1/p}<\frac{\varepsilon}{2\|T\|_q}
$$
where for $p=\infty$, the left-hand side becomes $\sup_{n>N} \lambda_n^\downarrow(C^+)=\lambda_{N+1}^\downarrow(C^+)\,$. 

Either way, associated to this $N$ one can choose $K\geq N$ such that the first $N$ largest eigenvalues of $T^+$ are
listed  in $(\lambda_j^\downarrow(\Pi_KT^+\Pi_K))_{j\in\mathbb N}$ and thus $\lambda_j^\downarrow(T^+)=\lambda_j^\downarrow(\Pi_KT^+\Pi_K)$ for all $j=1,\ldots,N$. Putting things together and using H\"older's inequality yields
\begin{align*}
\Big| \sum\nolimits_{j=1}^K \lambda_j^\downarrow(C^+) &\lambda_{j}^\downarrow(\Pi_KT^+\Pi_K)-\sum\nolimits_{j=1}^\infty\lambda_j^\downarrow(C^+)\lambda_j^\downarrow(T^+)\Big|\\
&=\Big| \sum\nolimits_{j=N+1}^K \lambda_j^\downarrow(C^+)\lambda_{j}^\downarrow(\Pi_KT^+\Pi_K)-\sum\nolimits_{j=N+1}^\infty\lambda_j^\downarrow(C^+)\lambda_{j}^\downarrow(T^+)\Big|\\
&\leq2\|T^+\|_q \Big(\sum\nolimits_{j=N+1}^\infty\big(\lambda_j^\downarrow(C^+)\big)^p\Big)^{1/p}<2\|T\|_q\frac{\varepsilon}{2\|T\|_q}=\varepsilon\,.
\end{align*} 
The case of $C^-,T^-$ as well as the infimum-estimate are shown analogously which concludes the proof.
\end{proof}

Therefore if $C,T$ are self-adjoint (i.e.~$W_C(T)\subseteq\mathbb R$), a path-connectedness argument similar to the proof of Lemma \ref{lemma_S_C_finite_rank} shows $(a,b)\subseteq W_C(T)\subseteq [a,b]$ with $a$ ($\leq 0$) given by \eqref{eq:inf_WCT} and $b$ ($\geq 0$) given by \eqref{eq:sup_WCT}. In particular, $\overline{W_C(T)}=[a,b]$.

\begin{acknowledgements} This work was supported 
by the Bavarian excellence network {\sc enb}
via the \mbox{International} PhD Programme of Excellence
{\em Exploring Quantum Matter} ({\sc exqm}).
\end{acknowledgements}

%

\section{APPENDIX}\label{appendix}
\renewcommand{\thesubsection}{\Alph{subsection}}
\subsection{PROOF OF THEOREM \ref{theorem_3}}\label{app_A}

The overall idea is to transfer properties of $W_C(T)$ from finite to infinite dimensions via
the set convergence introduced in Section \ref{sec:set_conv}. However, we first need two auxiliary
results to characterize the star-center of $\overline{W_C(T)}$ later on.

\begin{lemma}\label{lemma_2b}
Let $T\in\mathcal K(\mathcal X)$ and $(e_k)_{k\in\mathbb N}$ be any orthonormal system in $\mathcal X$. Then\vspace{2pt}
\begin{itemize}
\item[(a)]
$
\displaystyle
\sum\nolimits_{k=1}^n|\langle e_k,Te_k\rangle|
\leq\sum\nolimits_{k=1}^n s_k(T)
$
for all $n\in\mathbb N$ and\vspace{6pt}
\item[(b)]
$\lim_{k \to \infty} \langle e_k,Te_k\rangle = 0\,.$
\end{itemize}
\end{lemma}

\begin{proof}
(a) Consider a Schmidt decomposition $\sum\nolimits_{m=1}^\infty s_m(T)\langle f_m,\cdot\rangle g_m$ of $T$ so
\begin{align*}
\sum\nolimits_{k=1}^n|\langle e_k,Te_k\rangle|\leq \sum\nolimits_{m=1}^\infty s_m(T)\Big(\sum\nolimits_{k=1}^n |\langle e_k,f_m\rangle\langle g_m,e_k\rangle|\Big)\,.
\end{align*}
Defining $\lambda_m:= \sum_{k=1}^n |\langle e_k,f_m\rangle\langle g_m,e_k\rangle| $ for all $m\in\mathbb N$, using Cauchy-Schwarz and Bessel's inequality one gets
\begin{align*}
\lambda_m\leq \Big(\sum\nolimits_{k=1}^n |\langle e_k,f_m\rangle|^2\Big)^{1/2}\Big(\sum\nolimits_{k=1}^n |\langle g_m,e_k\rangle|^2\Big)^{1/2}\leq 1
\end{align*}
for all $m\in\mathbb N$. On the other hand, said inequalities also imply
\begin{align*}
\sum_{m=1}^\infty \lambda_m&\leq \sum_{k=1}^n \Big(\sum_{m=1}^\infty |\langle e_k,f_m\rangle|^2\Big)^{1/2}\Big(\sum_{m=1}^\infty |\langle g_m,e_k\rangle|^2\Big)^{1/2}\leq \sum_{k=1}^n \|e_k  \|^2=n\,.
\end{align*}
Hence, because $(s_m(T))_{m \in \mathbb N}$ is decreasing by construction, an upper bound of 
$\sum\nolimits_{m=1}^\infty s_m(T)\lambda_m$ is obtained by choosing $\lambda_1=\ldots=\lambda_n=1$
and $\lambda_j=0$ whenever $j>n$. This shows the 
desired inequality. A proof of (b) can be found, e.g., in \cite[Lemma 16.17]{MeiseVogt97en}.
\end{proof}

\begin{lemma}\label{lemma_0_conv}
Let $C\in\mathcal B^p(\mathcal H)$ with $p\in( 1,\infty]$ and let $q\in[1,\infty)$ such that $p,q$ are conjugate. Furthermore, let $(e_n)_{n\in\mathbb N}$ be any orthonormal system in
$\mathcal H$. Then
\begin{align*}
\lim_{n\to\infty}\frac{1}{n^{1/q}}\sum\nolimits_{k=1}^n\langle e_k,Ce_k\rangle=0\,.
\end{align*}
\end{lemma}
\begin{proof}
First, let $p=\infty$, so $q=1$. As $C$ is compact, by Lemma \ref{lemma_2b} (b) one has
$\lim_{k\to\infty}\langle e_k,Ce_k\rangle=0$, hence the sequence of arithmetic means converges
to zero as well. Next, let $p\in(1,\infty)$ and $\varepsilon>0$. Moreover, we assume w.l.o.g.~$C\neq 0$
so $s_1(C) = \Vert C\Vert\neq 0$. As $C\in \mathcal B^p(\mathcal H)$, one can choose $N_1\in\mathbb N$
such that $\sum_{k=N_1+1}^\infty s_k(C)^p<\frac{\varepsilon^p}{2^p}$ and moreover $N_2\in\mathbb N$ such 
that $\frac{1}{n^{1/q}}<\frac{\varepsilon}{2\sum\nolimits_{k=1}^{N_1}s_k(C)}$ for all $n\geq N_2$. Then, for any 
$n\geq N:=\max\lbrace N_1+1,N_2\rbrace$, Lemma \ref{lemma_2b}
and H\"older's inequality yield the estimate
\begin{align*}
\Big|\frac{1}{n^{1/q}}\sum\nolimits_{k=1}^n\langle e_k,Ce_k\rangle\Big|&\leq \frac{1}{n^{1/q}}\sum\nolimits_{k=1}^{N_1}s_k(C)+\frac{1}{n^{1/q}}\sum\nolimits_{k=N_1+1}^{n}s_k(C)\\
&\leq \frac{1}{n^{1/q}}\sum_{k=1}^{N_1}s_k(C)+\Big(\sum_{k=N_1+1}^{n} s_k(C)^p \Big)^{1/p}\Big(\sum_{k=N_1+1}^{n} \frac{1}{n} \Big)^{1/q}\\
&<\frac{\varepsilon}{2}+\Big(\sum\nolimits_{k=N_1+1}^{\infty} s_k(C)^p \Big)^{1/p}\Big( \frac{n-N_1}{n} \Big)^{1/q}\leq\varepsilon\,.\qedhere
\end{align*}
\end{proof}

What we also need is some mechanism to associate bounded operators on $\mathcal H$
with matrices. In doing so, let $(e_n)_{n\in\mathbb N} $ be some orthonormal basis of $\mathcal H$ and
let $(\hat e_i)_{i=1}^n$ be the standard basis of $\mathbb C^n$. For any $n\in\mathbb N$ we define $\Gamma_n:\mathbb C^n\to \mathcal H$, $\hat{e_i}\mapsto \Gamma_n(\hat e_i):=e_i$ and its linear extension to all of $\mathbb C^n$. With this, let 
\begin{align}\label{cut_out_operator}
[\;\cdot\;]_n:\mathcal B(\mathcal H)\to\mathbb C^{n\times n},\qquad A\mapsto [A]_n:=\Gamma_n^\dagger A\Gamma_n
\end{align}
be the operator which ``cuts out'' the upper $n\times n$ block of (the matrix representation of) $A$ 
with respect to $(e_n)_{n\in\mathbb N} $. The key result now is the following:

\begin{proposition}\label{lemma_6}
Let $C\in\mathcal B^p(\mathcal H)$, $T\in\mathcal B^q(\mathcal H)$ with $p,q\in [1,\infty]$ conjugate be given. Furthermore, let $(e_n)_{n\in\mathbb N}$ and $(g_n)_{n\in\mathbb N}$ be arbitrary orthonormal bases of $\mathcal H$. Then
\begin{align*}
\lim_{n\to\infty}W_{[C]^e_{2n}}([T]^g_{2n})=\overline{W_C(T)}
\end{align*}
where $[\,\cdot\,]_k^e$ and $[\,\cdot\,]_k^g$ are the maps given by \eqref{cut_out_operator} with respect to $(e_n)_{n\in\mathbb N}$ and $(g_n)_{n\in\mathbb N}$, respectively. Moreover, if $C$ are $T$ both are normal then
\begin{align*}
\lim_{n\to\infty}P_{[C]^e_n}([T]^g_n)= \overline{P_C(T)}\,.
\end{align*}
where $(e_n)_{n\in\mathbb N}$ and $(g_n)_{n\in\mathbb N}$ are the orthonormal bases of $\mathcal H$ which diagonalize $C$ and
$T$, respectively.
\end{proposition}

\begin{proof}
For $p=1,q=\infty$ (or vice versa) proofs are given in \cite[Thm.~3.1 \& 3.6]{dirr_ve} which can be adjusted 
to $p,q\in(1,\infty)$ by minimal modifications.
\end{proof}

With these preparations we are ready for proving our main result about the $C$-numerical range of Schatten-class operators.

\begin{proof}[Proof of Theorem \ref{theorem_3}]
(a): For arbitrary orthonormal bases $(e_n)_{n\in\mathbb N}$, $(g_n)_{n\in\mathbb N}$ of $\mathcal H$ as well as any $n\in\mathbb N$, it is readily verified that
\begin{align*}
\frac{\operatorname{tr}([C]^e_{2n})\operatorname{tr}([T]^g_{2n})}{2n} &=\frac{\operatorname{tr}([C]^e_{2n})}{(2n)^{1/q}}\frac{\operatorname{tr}([T]^g_{2n})}{(2n)^{1/p}}\\
&=\Big( \frac{1}{(2n)^{1/q}}\sum\nolimits_{j=1}^{2n} \langle e_j,Ce_j\rangle \Big)\Big( \frac{1}{(2n)^{1/p}}\sum\nolimits_{j=1}^{2n} \langle g_j,Tg_j\rangle \Big)\,.
\end{align*}
Both factors converge and, by Lemma \ref{lemma_0_conv}, at least one of them goes to $0$ as $n\to\infty$. 
Moreover, $W_{[C]^e_{2n}}([T]^g_{2n})$ is star-shaped with respect to 
$(\operatorname{tr}([C]^e_{2n})\operatorname{tr}([T]^g_{2n})/(2n)$ for all $n\in\mathbb N$,
cf.~\cite[Thm.~4]{TSING-96}. Because Hausdorff convergence preserves star-shapedness \cite[Lemma 2.5 (d)]{dirr_ve}, Proposition \ref{lemma_6}
implies that $\overline{W_{C}(T)}$ is star-shaped with respect to $0 \in \mathbb{C}$.\medskip

For what follows let $(e_n)_{n\in\mathbb N},(g_n)_{n\in\mathbb N}$ be the orthonormal bases of $\mathcal H$ which diagonalize $C$ and $T$, respectively.\medskip

(b): W.l.o.g. let $C$ be normal with collinear eigenvalues. Since $C$ is compact (i.e.~its eigenvalue sequence is a null sequence) there exists
$\phi\in[0,2\pi)$ such that $e^{i\phi}C$ is self-adjoint and by Proposition \ref{lemma_6} we obtain
\begin{align*}
\overline{W_C(T)}=\overline{W_{e^{i\phi}C}(e^{-i\phi}T)}=\lim_{n\to\infty} W_{[e^{i\phi}C]_{2n}^e}([e^{-i\phi}T]_{2n}^e)\,.
\end{align*}
Moreover, as $[e^{i\phi}C]_{2n}^e\in\mathbb C^{2n\times 2n}$ is hermitian for all $n\in\mathbb N$ we conclude that $W_{[e^{i\phi}C]_{2n}^e}([e^{-i\phi}T]_{2n}^e)$ is convex, cf.~\cite{article_poon}. The fact that Hausdorff convergence preserves convexity \cite[Lemma 2.5 (c)]{dirr_ve} then yields the desired result.\medskip

(c): The inclusion $P_C(T)\subseteq W_C(T)$ is shown exactly like \cite[Thm.~3.4--first inclusion]{dirr_ve}. For the second inclusion, we note that by assumption $[C]^e_n$ and $[T]^g_n$ are diagonal and thus normal for all $n\in\mathbb N$. Hence \cite[Coro.~2.4]{Sunder82} tells us
\begin{align}\label{eq:wc_pc_incl}
W_{[C]_{2n}^e}([T]_{2n}^g)\subseteq \operatorname{conv}(P_{[C]_{2n}^e}([T]_{2n}^g))
\end{align}
for all $n\in\mathbb N$. Using that Hausdorff convergence preserves inclusions \cite[Lemma 2.5 (a)]{dirr_ve}, \eqref{eq:wc_pc_incl} together with Proposition \ref{lemma_6} yields
\begin{align*}
W_C(T)\subseteq \overline{W_C(T)}=\lim_{n\to\infty} W_{[C]_{2n}^e}([T]_{2n}^g)\subseteq \lim_{n\to\infty} \operatorname{conv}( P_{[C]_{2n}^e}([T]_{2n}^g) )= \operatorname{conv}(\overline{P_C(T)})\,.
\end{align*}
Finally, applying the closure and the convex hull to the inclusions $P_C(T) \subseteq W_C(T)$
yields $\operatorname{conv}(\overline{P_C(T)})\subseteq\operatorname{conv}(\overline{W_C(T)})=\overline{W_C(T)}$, where the last equality is due to (b), and thus 
$\overline{W_C(T)} = \operatorname{conv}(\overline{P_C(T)})$.
\end{proof}

\end{document}